\documentclass[11pt]{amsart}

\usepackage[dvipdfm]{hyperref}
\usepackage{mathptmx}
\usepackage{amsmath}
\usepackage{amssymb}
\usepackage{amsthm}
\usepackage{array}
\usepackage[all,tips]{xy}
\usepackage{enumerate}
\usepackage{graphicx}
\usepackage{lmodern}
\usepackage{mathrsfs}

\newtheorem{theorem}{Theorem}

\newtheorem{prop}[theorem]{Proposition}
\newtheorem{cor}[theorem]{Corollary}

\theoremstyle{definition}

\newtheorem{rmk}[theorem]{Remark}

\DeclareMathOperator{\GL}{GL}

\DeclareMathOperator{\Spin}{Spin}

\DeclareMathOperator{\PO}{PO}

\DeclareMathOperator{\PSO}{PSO}
\DeclareMathOperator{\PSL}{PSL}

\newcommand{\Z}{\mathbb Z}
\newcommand{\Q}{\mathbb Q}
\newcommand{\R}{\mathbb R}
\newcommand{\C}{\mathbb C}

\renewcommand{\P}{\mathbf P}
\newcommand{\f}{\mathfrak f}

\newcommand{\V}{V}
\newcommand{\Vf}{\V_{\mathrm{f}}}

\newcommand{\tB}{\mathrm{B}}

\newcommand{\tD}{\mathrm{D}}

\newcommand{\Hy}{\mathbf H}
\newcommand{\Isom}{\mathrm{Isom}}

\newcommand{\GG}{\mathbf G}
\newcommand{\MM}{\mathbf M}
\newcommand{\cM}{\mathcal M}

\newcommand{\bs}{\backslash}

\newcommand{\Cox}{\xymatrix{ *={\bullet} \ar@{-}[r]^{5}&
				*={\bullet} \ar@{-}[r] & *={\bullet} \ar@{-}[r] & *={\bullet}
				\ar@{-}[r]  & *={\bullet} 
			}}

\begin{document}

\begin{abstract}
We prove that for $n>4$ there is no compact arithmetic hyperbolic
$n$-manifold whose Euler characteristic has absolute value equal to 2.
In particular, this shows the nonexistence of arithmetically defined
hyperbolic rational homology $n$-sphere with $n$ even different than 4.
\end{abstract}

\title{On compact hyperbolic manifolds of Euler characteristic two}

\author{Vincent Emery}
\thanks{Supported by Swiss National Science Foundation, Project number
  {\tt PA00P2-139672}}
\address{
Department of Mathematics\\
Stanford University\\
California 94305\\
USA
}
\email{vincent.emery@gmail.com}

\date{\today}
\dedicatory{Dedicated to the memory of Colin Maclachlan}

\maketitle

\section{Main result and discussion}
\label{intro}

\subsection{Smallest hyperbolic manifolds}
\label{ss:intro-thm}

Let $\Hy^n$ be the hyperbolic $n$-space.
By a \emph{hyperbolic $n$-manifold} we mean an orientable manifold
$M = \Gamma\bs \Hy^n$, where $\Gamma$ is a torsion-free discrete subgroup
$\Gamma \subset \Isom^+(\Hy^n)$. 
The set of volumes of hyperbolic $n$-manifolds being well ordered, it is
natural to try to determine for each dimension $n$ the hyperbolic manifolds
of smallest volume.
For $n=3$ this problem has recently been solved
in \cite{GMM10}, the smallest volume being achieved by a unique compact manifold,
the Weeks manifold. When $n$ is even the volume is proportional to the
Euler characteristic, and this allows to formulate the problem in terms
of finding the hyperbolic manifolds $M$ with smallest $|\chi(M)|$. In
particular this observation solves the problem in the case of surfaces.
For $n>3$, noncompact hyperbolic $n$-manifolds $M$ with $|\chi(M)| = 1$ 
have been found for $n = 4, 6$ \cite{ERT10}.

In the present paper we consider the case of compact
manifolds of even dimension. In particular, such manifolds have even
Euler characteristic (see \cite[Theorem 1.2]{KelZeh01}).  We restrict ourselves to the case of
\emph{arithmetic} manifolds, where Prasad's formula \cite{Pra89} can be used to study
volumes. We complete the proof of the following result.

\begin{theorem}
 \label{thm:non-existence}
 Let $n>5$. There is no compact arithmetic manifold $M =
 \Gamma\bs\Hy^n$ with $|\chi(M)| = 2$. 
\end{theorem}

The result for $n > 10$ already follows from the work of Belolipetsky
\cite{Belo04,Belo07}, also based on Prasad's volume formula. More
precisely, Belolipetsky determined the smallest Euler characteristic
$|\chi(\Gamma)|$ for arithmetic orbifold quotients $\Gamma\bs\Hy^n$ ($n$
even). This smallest value grows fast with the dimension $n$, and for
compact quotients we have $|\chi(\Gamma)| > 2$ for $n>10$.
That the result of nonexistence holds
for $n$ high enough is already a consequence of Borel-Prasad's general
finiteness result \cite{BorPra89}, which was the first application of
Prasad's formula.  The proof of Theorem \ref{thm:non-existence} for
$n=6, 8, 10$ requires a more precise analysis of the Euler
characteristic of arithmetic subgroups $\Gamma \subset \PO(n,1)$,
and in particular of the special values of Dedekind zeta functions 
that appear as factors of $\chi(\Gamma)$.



For $n = 4$, the corresponding problem is not solved, but there is the
following result \cite{Belo07}.
\begin{theorem}[Belolipetsky]
 If $M = \Gamma\bs\Hy^4$ is a compact arithmetic manifold with
 $\chi(M) \le 16$, then $\Gamma$ arises as a (torsion-free) subgroup of the following
 hyperbolic Coxeter group:
 \begin{align}
 	W_1 &= \;\; \Cox
   \label{W1}
 \end{align}
  \label{thm:dim-4}
\end{theorem}

An arithmetic (orientable) hyperbolic $4$-manifold of 
Euler characteristic $16$ has been first constructed by
Conder and Maclachlan in \cite{CondMac05}, using the
presentation of $W_1$ to obtain a torsion-free subgroup with the help of
a computer. Further examples with $\chi(M) = 16$ have been obtained by Long
in \cite{Long08} by considering a homomorphism from $W_1$ onto the finite simple
group $\mathrm{PSp}_4(4)$.



\subsection{Hyperbolic homology spheres}
\label{ss:homology-spheres}

Our original motivation for Theorem \ref{thm:non-existence} was the
problem of existence of hyperbolic  homology spheres. A \emph{homology
$n$-sphere} (resp. \emph{rational} homology $n$-sphere) is a
$n$-manifold $M$ that possesses the same integral (resp. rational)
homology as the $n$-sphere $S^n$. This forces $M$ to be compact and
orientable.  


Rational homology $n$-spheres $M$ have $\chi(M) = 2$ if $n$ is even. 
On the other hand, for $M = \Gamma\bs\Hy^n$ with $n = 4k+2$ we have
$\chi(M) < 0$ (cf. \cite[Proposition 23]{Serr71}), and this exclude the possibility of hyperbolic
rational homology spheres for those dimensions. For $n$ even,
Wang's finiteness theorem \cite{Wang72} implies that there is only a
finite number of hyperbolic rational homology $n$-spheres. 
Theorem \ref{thm:non-existence} shows the nonexistence of arithmetic rational
homology spheres for $n>5$ even. 

For odd dimensions, $\chi(M) = 0$ and \emph{a priori} the volume is not
a limitation for the existence of hyperbolic (rational) homology
spheres. In fact, an infinite tower of covers by hyperbolic integral homology
$3$-spheres has been constructed by Baker, Boileau and Wang in
\cite{BBW01}. In \cite{CalDun06} Calegari and Dunfield constructed an infinite tower of hyperbolic rational
homology $3$-spheres that are arithmetic and obtained by congruence
subgroups. Note that a recent conjecture of Bergeron and Venkatesh
predicts a lot of torsion in the homology groups of such a ``congruence tower'' of arithmetic
$n$-manifolds with $n$ odd \cite{BergVenka}.

\subsection{Locally symmetric homology spheres}
\label{ss:locally-symmetric}

Instead of considering hyperbolic homology spheres, one can more
generally look for homology spheres that are locally isometric to a
given symmetric space of nonpositive nonflat sectional curvature.
Such a symmetric space $X$ is called of noncompact type, and
it is classical that $X$ can be written as $G/K$, where $G$ is a
connected real semisimple Lie group with trivial center with $K \subset G$ a
maximal compact subgroup.  Moreover, $G$ identifies as a finite index
subgroup in the group of isometries of $X$ (of index two if $G$ is
simple).

Let us explain why the case $X = \Hy^n$ is the main source
of locally symmetric rational homology spheres (among $X$ of noncompact
type). Let $M$ be a compact orientable manifold locally isometric to
$X$. Then $M$ can be written as $\Gamma\bs X$, where $\Gamma \cong \pi_1(M)$ is
a discrete subgroup of
isometries of $X$. We will suppose that $\Gamma \subset G$, for $G$ as
above.  Let $X_u$ be the compact dual of $X$. We have the
following general result (see \cite[Sections 3.2 and 10.2]{Bor74}).

\begin{prop}
  \label{prop:inj-Hj}
  For each $j$ there is an injective homomorphism $H^j(X_u,\C) \to
  H^j(\Gamma\bs X,\C)$.
\end{prop}

In particular, if $\Gamma\bs X$ is a rational homology sphere, then so is
$X_u$. Note that the compact dual of $X = \Hy^n$ is the genuine sphere
$S^n$.  By looking at the classification of compact symmetric spaces,
Johnson showed the following in \cite[Theorem 7]{Johns82}.
\begin{cor}
  \label{cor:restrict-hyperbolic}
If $M = \Gamma \bs X$ is a rational homology $n$-sphere with $\Gamma \subset
G$, then $X$ is either the hyperbolic $n$-space $\Hy^n$ (with $n \neq
4k+2$), or $X = \PSL_3(\R)/\PSO(3)$ (which has dimension $5$). 
\end{cor}


Proposition \ref{prop:inj-Hj} shows that the correct problem to look at
-- rather than homology spheres -- is the existence of locally symmetric spaces $\Gamma\bs X$ with the same (rational) homology as the compact dual
$X_u$.  
When $X$ is the complex hyperbolic plane $\Hy^2_\C$, the compact
dual is the projective plane $\P_\C^2$, and the quotients $\Gamma\bs X$
are compact complex surfaces called \emph{fake projective planes}. 
Their classification was recently obtained by the work of Prasad--Yeung
\cite{PraYeu07}, together with Cartwright--Steger \cite{CarSt10} who performed the
necessary computer search. Later, Prasad and Yeung also considered 
the problem of the existence of more general arithmetic fake Hermitian spaces
\cite{PraYeu09,PraYeu12}.

The present paper uses the same methodology as in Prasad and Yeung's work,
the main ingredient being the volume formula. 

\subsection*{Acknowledgements} It is a pleasure to thank Gopal Prasad,
who suggested this research project.

\section{Proof of Theorem \ref{thm:non-existence}}
\label{sec:proof-thm1}

Let $G = \PO(n,1)^\circ \cong \Isom^+(\Hy^n)$,
and consider the universal covering $\phi: \Spin(n,1) \to G$. 
For our purpose it will be easier to work with lattices in $\Spin(n,1)$.
A lattice $\overline{\Gamma} \subset G$ is arithmetic exactly 
when $\Gamma = \phi^{-1}(\overline{\Gamma})$ is an arithmetic subgroup of
$\Spin(n,1)$.
Since the covering $\phi$ is twofold, we have  $\chi(\Gamma) = \frac{1}{2}
\chi(\overline{\Gamma})$, where $\chi$ is the Euler characteristic in
the sense of C.T.C.~Wall. In particular, if $M = \overline{\Gamma}\bs \Hy^n$ 
is a manifold with $|\chi(M)| = 2$, then $|\chi(\Gamma)| = 1$.
Thus, Theorem \ref{thm:non-existence} is an obvious consequence of the 
following proposition. The proof relies on the description of arithmetic
subgroups with the help of Bruhat-Tits theory, as done for instance in
\cite{BorPra89} and \cite{Pra89}. An introduction can be found in
\cite{EmePhD}.  We also refer to \cite{Tits79} for the needed facts from 
Bruhat-Tits theory. 

\begin{prop}
Let $n>4$.
 There is no cocompact arithmetic lattice $\Gamma \subset \Spin(n,1)$ such that
 $\chi(\Gamma)$ is a reciprocal integer, i.e., such that $\chi(\Gamma) = 1 / q$
 for some $q \in \Z$. 
\end{prop}

\begin{proof}
  We can assume that $n$ is even.
  Let $\Gamma \subset  \Spin(n,1)$ be a cocompact lattice. 
 Clearly, it suffices to prove the proposition for $\Gamma$ maximal. In
 this case, $\Gamma$ can be written as the normalizer $\Gamma =
 N_{\Spin(n,1)}(\Lambda)$ of some {\em principal} arithmetic subgroup
 $\Lambda$ (see \cite[Proposition 1.4]{BorPra89}).
 By definition, there exists a number field $k \subset \R$ and a
 $k$-group $\GG$ with $\GG(\R) \cong \Spin(n,1)$ such that $\Lambda =
 \GG(k) \cap \prod_{v \in \Vf} P_v$, for some coherent collection
 $(P_v)_{v\in\Vf}$ of
 parahoric subgroups $P_v \subset \GG(k_v)$ (indexed by the set $\Vf$ of
 finite places of $k$). 
 It follows from the classification of algebraic groups (cf.
 \cite{Tits66}) that $\GG$ is of type $\tB_r$ with $r = n / 2\; (> 2)$,
 the field $k$ is totally real, and (using Godement's criterion) $k \neq
 \Q$. Let us denote by $d$ the degree $[k:\Q]$.

 Let $T \subset \Vf$ be
 the set of places where $P_v$ is not hyperspecial. By Prasad's volume
 formula (see \cite{Pra89} and \cite[Section 4.2]{BorPra89}), we have:
 \begin{align}
   |\chi(\Lambda)| &= 2 |D_k|^{r^2+r / 2}
   C(r)^d
   \prod_{j=1}^r \zeta_k(2j) \prod_{v \in T} \lambda_v,
   \label{eq:volume}
 \end{align}
 with $D_k$ (resp. $\zeta_k$) the discriminant (resp. Dedekind zeta
 function) of $k$; the constant $C(r)$ is given by
 \begin{align}
  C(r) =  \prod_{j=1}^r\frac{(2j-1)!}{(2\pi)^{2j}};
   \label{eq-Cr}
 \end{align}
 and each $\lambda_v$ is given by the formula
 \begin{align}
   \lambda_v &= \frac{1}{(q_v)^{(\dim\cM_v - \dim \MM_v) / 2}}\;
   \frac{|\cM(\f_v)|}{|\MM_v(\f_v)|},
   \label{eq:lambda}
 \end{align}
 where $\f_v$ is the residue field of $k_v$, of size $q_v$, and the
 reductive $\f_v$-groups $\MM_v$ and $\cM_v$ associated with $P_v$ are those
 described in \cite{Pra89}. By definition $\cM_v$ is semisimple of type
 $\tB_r$.

 \begin{table}
   \centering
   \begin{tabular}{llc}
     $\GG / k_v$ & isogeny type of $\MM_v$   & $\lambda_v$\\[9pt] \hline
     \mbox{split: } & $\tB_{r-1} \times (\mbox{split }\GL_1)$  & $\frac{q^{2r}-1}{q-1}$\\[9pt]
    &$\tD_i \times \tB_{r-i}$  ($i=2,\dots,r-1$) & $
     \frac{(q^i+1)\prod_{k=i+1}^r(q^{2k}-1)}{\prod_{k=1}^{r-i}(q^{2k}-1)}$\\[9pt]
    &$^1\tD_r$ & $q^r + 1$\\[9pt]
    \mbox{non-split: } & $\tB_{r-1} \times (\mbox{nonsplit } \GL_1)$ & $\frac{q^{2r}-1}{q+1}$\\[9pt]
    & $^2\tD_{i+1} \times \tB_{r-i-1}$ ($i=1,\ldots,r-2$) & $
    \frac{(q^{i+1}-1)\prod_{k=i+2}^r(q^{2k}-1)}{\prod_{k=1}^{r-i-1}(q^{2k}-1)}$\\[9pt]
    & $^2\tD_r$ & $q^r -1$\\[9pt]
   \end{tabular}
   \caption{$\lambda_v$ for $P_v$ of maximal type}
   \label{tab:lambda-fac}
 \end{table}

 A necessary condition for  $\Gamma = N_{\GG(\R)}(\Lambda)$ to be maximal is that 
each $P_v$ defining $\Lambda$ has maximal type in the sense of
\cite{RyzCher97}. We list in Table \ref{tab:lambda-fac} the factors
$\lambda_v$ corresponding to parahoric sugroups $P_v$ of maximal types
(to improve the readability we set $q_v=q$ in the formulas). This list
of maximal type and the formulas for $\lambda_v$ are essentially the same
as in \cite[Table 1]{Belo04}: the only difference is a factor $2$ in the
denominator of some $\lambda_v$, which can be explained from the fact
that Belolipetsky did not work with $\GG$ simply connected. 

From \cite[Section 5]{BorPra89} (cf. also \cite[Chapter 12]{EmePhD}) we can 
deduce that the index 
$[\Gamma : \Lambda]$ of $\Lambda$ in its normalizer has the following
property:
 \begin{align}
   [\Gamma : \Lambda] \mbox{ divides } h_k 2^{d} 4^{\#T}.
   \label{eq:index}
 \end{align}
Moreover, a case by case analysis of the possible factor $\lambda_v$
shows that $\lambda_v > 4$, so that $4^{-\#T} \prod_{v \in T} \lambda_v
\ge 1$ (with equality exactly when $T$ is empty).
We thus have the following  lower bound for the Euler
characteristic of any maximal arithmetic subgroup $\Gamma \subset
\Spin(n,1)$:
\begin{align}
  |\chi(\Gamma)| &\ge \frac{2}{h_k} \left(\frac{C(r)}{2}\right)^d
  |D_k|^{r^2 + r / 2 } \zeta_k(2) \cdots \zeta_k(2r)
  \label{eq:first-bd-Gamma}
\end{align}
We make use of the following upper bound for the class number (see for
instance \cite[Section 7.2]{BelEme}):
\begin{align}
 h_k &\le 16 \left(\frac{\pi}{12}\right)^d |D_k|,
  \label{eq:bound-hk}
\end{align}
which together with the basic inequality $\zeta_k(2j) > 1$
transforms \eqref{eq:first-bd-Gamma} into 
\begin{align}
  |\chi(\Gamma)| &> \frac{1}{8} \left(\frac{6 \cdot C(r)}{\pi}\right)^d
  |D_k|^{r^2 + r / 2 -1}.
  \label{eq:bd-no-hk}
\end{align}
Moreover, according to \cite[Table 4]{Odl90}, we have that for a degree
$d \ge 5$ the discriminant of $k$ is larger than $(6.5)^d$. With this
estimates we can check that for $r \ge 3$ and $d \ge 5$ we have
$|\chi(\Gamma)| > 1$. For the lower degrees, if we suppose that
$|\chi(\Gamma)| \le 1$, we obtain upper bounds for $|D_k|$ from Equation
\eqref{eq:bd-no-hk}. This upper bounds exclude the existence of such a
$\Gamma$ for $r \ge 6$ (which is already clear from the work of
Belolipetsky \cite{Belo04}). For $r=3$ (where the bounds are the worst) we obtain
the following:
\begin{align*}
  d = 2: &\quad |D_k| \le 28;\\
  d = 3: &\quad |D_k| \le 134;\\
  d = 4: &\quad |D_k| \le 640.
\end{align*}
From the existing tables of number fields (e.g.,
\cite{Bordeaux_data,Qaos}) we can list the possibilities this leaves us
for $k$. We find that no field with $d=4$ can appear, and for $d=2,3$
all possibilities have class number $h_k = 1$. Using Equation
\eqref{eq:bound-hk} with $h_k = 1$ we then improve the upper bounds for
$|D_k|$ and thus shorten the list of possible fileds. For $r=5$ only
$|D_k| = 5$ arises, and for $r=4$ we have $|D_k| \le 11$ (the
possibility $d = 3$ is excluded here). For $r=3$, we are left with
$|D_k| \le 20$ when $d = 2$, and  $|D_k| = 49$ or $81$ when $d=3$.

With $h_k = 1$, using the functional equation for $\zeta_k$ and the property
\eqref{eq:index} for the index $[\Gamma:\Lambda]$, we can express the
Euler characteristic of $\Gamma$ as 
\begin{align}
  |\chi(\Gamma)| &= \frac{1}{2^a} \prod_{v \in T} \lambda_v
  \;\prod_{j=1}^r |\zeta_k(1-2j)|
 \label{eq:EP-rational}
\end{align}
for some integer $a$. The special values $\zeta_k(1-2j)$, which are
rational by the Klingen-Siegel theorem, can be computed with the software
Pari/GP (cf. Remark \ref{rmk:special-val}).
We list in Table \ref{tab:special-values} the
values we need. We check that for
every field $k$ under consideration a prime factor $>2$ appear in the
numerator of  
the product $\prod_{j=1}^r |\zeta_k(1-2j)|$.
A direct computation for $r=3,4,5$ shows that the
formula in Table \ref{tab:lambda-fac} for each factor $\lambda_v$ is actually 
given by a polynomial in $q$ (this seems to hold for any $r$).
In particular, we always have $\lambda_v \in \Z$,  and  we conclude from
\eqref{eq:EP-rational} that 
$|\chi(\Gamma)|$ cannot be a reciprocal integer.  
\begin{table}
  \centering
  \begin{tabular}{ccccccc}
   degree & $|D_k|$ & $\zeta_k(-1)$ & $\zeta_k(-3)$ & $\zeta_k(-5)$ &
   $\zeta_k(-7)$ & $\zeta_k(-9)$ \\ \hline
    $d=2$ & 5 & 1/30 & 1/60  & 67/630 & 361/120 & 412751/1650 \\
    &8 & 1/12 & 11/120 & 361/252& 24611/240& \\
    &12 & 1/6 & 23/60 & 1681/126& & \\
    &13 & 1/6 & 29/60 &  33463/1638 & & \\
    &17 & 1/3 & 41/30 & 5791/63 & & \\[2pt]
   $d=3$ & 49 & -1/21 & 79/210 & -7393/63 & & \\
   &81 & -1/9 & 199/90 & -50353/27 & & \\[8pt]
  \end{tabular}
  \caption{Special values of $\zeta_k$}
  \label{tab:special-values}
\end{table}
\end{proof}

\begin{rmk}
 \label{rmk:special-val}
 The function \texttt{zetak} in Pari/GP allows to obtain approximate
 values for $\zeta_k(1-2j)$. On the other hand the size of the
 denominator of the product $\prod_{j=1}^m \zeta_k(1-2j)$ can be bounded
 by the method described in \cite[Section 3.7]{Serr71}.
 By recursion on $m$, this allows to ascertain that the values
 $\zeta_k(1-2j)$ correspond exactly to the fractions given in Table
 \ref{tab:special-values}.
\end{rmk}

\begin{rmk}
 The fact that for $|D_k| = 5$ the value  $\zeta_k(-1) \zeta_k(-3)$  has
 trivial numerator explains why the proof fails for
 $n = 4$ (i.e., $r=2$). And indeed  there is a principal
 arithmetic subgroup $\Gamma \subset \Spin(4,1)$ with $|\chi(\Gamma)| =
 1 / 14400$ and whose image in $\Isom^+(\Hy^4)$ is contained as an index
 $2$ subgroup of the Coxeter group $W_1$. On the other hand, for $|D_k| > 5$ the appearance
 of a non-trivial numerator in $\zeta_k(-3)$ shows -- at least for the
 fields considered in Table \ref{tab:special-values} -- the
 impossibility of a $\Gamma$ defined over $k$ with $\chi(\Gamma)$ a
 reciprocal integer. This is the first step in Belolipetsky's proof of
 Theorem \ref{thm:dim-4}. 
\end{rmk}

\bibliographystyle{amsplain}
\bibliography{Euler2}

\end{document}